\documentclass{amsart}
\usepackage{amsfonts}
\usepackage{amsmath}
\usepackage{amssymb}
\usepackage{graphicx} 
\usepackage{version}

\theoremstyle{plain}
\newtheorem{theorem}{Theorem}
\newtheorem{cor}[theorem]{Corollary}

\theoremstyle{definition}

\newtheorem{question}{Question}
\newtheorem{conjecture}[question]{Conjecture}

\begin{document}

\title{The Average Size of a Connected Vertex Set of a $k$-connected Graph}
\author[A. Vince]{Andrew Vince}
\address{Department of Mathematics \\ University of Florida \\ Gainesville, FL 32611, USA}
\email{\tt  avince@ufl.edu} 
\thanks{This work was partially supported by a grant from the Simons Foundation (322515 to Andrew Vince).}

\maketitle

\begin{abstract}  
The topic is the average order $A(G)$ of a connected induced subgraph of a graph $G$.  This generalizes, to graphs in general, the average order of a subtree of a tree.  In 1984, Jamison proved that the average order, over all trees of order $n$, is minimized by the path $P_n$, the average being $A(P_n)=(n+2)/3$.  In 2018, Kroeker, Mol, and Oellermann conjectured that $P_n$ minimizes the average order over all connected graphs $G$ -
a conjecture that was recently proved.  In this short note we show that this lower bound can be improved if the connectivity of $G$ is known.  If $G$ is $k$-connected, then  
\[A(G) \geq \frac{n}2 \Bigg (1-  \frac{1}{2^k+1} \Bigg ).\]
\end{abstract}

\section{Introduction} \label{sec:intro}

 Although connectivity is a basic concept in graph theory, problems involving 
the enumeration of the connected induced subgraphs 
of a given graph have only recently received attention.  The topic of this paper is the average order of a connected induced subgraph 
of a graph.  Let $G$ be a connected finite simple graph with vertex set $V = \{1,2,\dots, n\}$, and let $U\subseteq V$. 
 The set $U$ is said to 
be a {\bf connected set} if the subgraph of $G$ induced by $U$ is connected.  Denote the collection of all connected sets, excluding the emptyset, by $\mathcal C = \mathcal C(G)$.  The number of connected sets in $G$ will be denoted by $N(G)$.   Let
\[S(G) = \sum_{U \in \mathcal C} |U|  \] 
be the sum of the sizes of the connected sets.
Further, let
\[ A(G) = \frac{S(G)}{N(G)}   \qquad  \qquad \text{and}  \qquad  \qquad   D(G) =\frac{ A(G)}{n} \]
denote, respectively, the average size of a connected set of $G$ and the proportion of vertices in an average size connected set.  The parameter $D(G)$ is referred to as the {\bf density} of connected sets of vertices.    The density
allows us to compare the  average size of connected sets of graphs of different orders.   If, for example, $G$ is the complete graph $K_n$, then $A(K_n)$ is the average size of a subset of an $n$-element set, which is $n/2$ (counting the empty set for simplicity), and the density is $1/2$.

Papers \cite{H,J,J2,MO,SO,V,WW,YY} on the average size and density of connected  sets of trees appeared
beginning with Jamison's 1984 paper \cite{J}.  The invariant $A(G)$, in this case, is the average order of a subtree of a tree.   Concerning lower bounds, Jamison proved that the average size of a subtree of a tree of order $n$ is minimized by the path $P_n$.  In particular $A(T) \geq (n+2)/3$ for all trees $T$ with equality only for $P_n$.  Therefore $D(T) > 1/3$ for all trees $T$.  Vince and Wang \cite{V} proved that if $T$ is a tree all of whose non-leaf vertices have degree at least three,  then $\frac 12 \leq D(T) < \frac34,$  both bounds being best possible.

Although results are known for trees, little was known until recently for graphs in general.  Kroeker, Mol, and Oellermann conjectured in their 2018 paper \cite{O} 
that the lower bound of Jamison for trees extends to graphs in general.  

\begin{conjecture} The $P_n$ minimizes the average size of a connected set over all connected graphs.  
\end{conjecture} 

In \cite{B}  Balodis, Mol, and Oellermann verified the conjecture for block graphs of order $n$, i.e., for graphs each maximal $2$-connected component of which is a complete graph.  The conjecture was proved recently for connecterd graphs in general  in \cite{V1} and independently shortly thereafter in \cite{H2}.   

\begin{theorem} \label{thm:main2}  If $G$ is a connected graph of order $n$, then 
\[A(G) \geq \frac{n+2}{3},\]
with equality if and only if $G$ is a path.  In particular, $D(G) > 1/3$ for all connected graphs $G$.  
\end{theorem}

In \cite{V3} it was conjectured that the lower bound of Vince and Wang for trees extends to graphs in general.  

\begin{conjecture} \label{conj:deg} If $G$ is a conectred graph all of whose vertices have degree at least $3$, then $D(G) \geq \frac12$.  
\end{conjecture}   

Kroeker, Mol, and Oellermann \cite{O} verified Conjecture~\ref{conj:deg} for connected cographs.  A {\it cograph} can be defined recursively: the one vertex graph is a cograph and, if $G$ and $H$ are cographs, then so is their disjoint union and their join. (The {\it join} is obtained by joining by an edge each vertex of $G$ to each vertex of $H$.)  Complete bipartite graphs are examples of cographs. Conjecture~\ref{conj:deg} remains open, but in this note we prove a lower bound on $D(G)$ close to $1/2$ if $G$ is highly connected.   More precisely:

\begin{theorem} \label{thm:main}
If $G$ is $k$-connected, then 
\[D(G) \geq \frac12 \Bigg (1-  \frac{1}{2^k+1} \Bigg ).\]
\end{theorem}

\section{Proof of Theorem~\ref{thm:main} } \label{sec:proof}

If $i\in V$, let $N(G,i), S(G,i)$, and $A(G,i)$ denote the number of connected sets in 
$G$ containing $i$, the sum of the sizes of all connected sets containing $i$, and the average size of a connected set containing $i$, respectively.  The following result appears in \cite[Corollary 3.2]{V1}.

\begin{theorem} \label{thm:v}
If $i\in V$ is any vertex of a connected graph $G$ of order $n$, then
\[A(G,i) \geq \frac{n+1}{2}.\]
\end{theorem} 

\begin{cor} \label{cor:1} If $G$ is a connected graph of order $n$, then
\[\sum_{U \in \mathcal C} |U|^2  \geq \Big (\frac{n+1}{2} \Big ) S(G).\]
\end{cor}

\begin{proof} From Theorem~\ref{thm:v} we have
\[S(G,i) \geq \frac{n+1}{2} N(i)\]
for all $i \in V$.  
Now count the number of pairs in the set $\{ (i, U) :  U \in \mathcal C, i \in U \subseteq V\}$ in two ways to obtain
\[ S(G) = \sum_{U\in \mathcal C} |U| = \sum_{i\in V} N(G,i).\]
Similarly
\[\sum_{U \in \mathcal C} |U|^2 =  \sum_{U \in \mathcal C} \sum_{i : i\in U} |U| 
= \sum_{i\in V} \sum_{U\in \mathcal C : i \in U} |U| 
= \sum_{i \in V} S(G,i) \geq \frac{n+1}{2} \sum_{i \in V} N(G,i)  =
\Big ( \frac{n+1}{2} \Big ) S(G).\]
\end{proof}


\begin{proof}[Proof of Theorem~\ref{thm:main}] The proof is by induction on $k$.  When $k=1$, the statement
is $D(G) \geq \frac13$, which follows from Theorem~\ref{thm:main2} in the
introduction. 

 Let
\[ a_k := \frac12 \Bigg (1-  \frac{1}{2^k+1} \Bigg ).\]
A straightforward calculation shows that
\[a_k = \frac{ 2 a_{k-1}}{2 a_{k-1}+1}.\]

Assume that the statement of Theorem~\ref{thm:main} 
holds for all $(k-1)$-connected graphs and
asume that $G$ is $k$-connected.  Let $\mathcal C ' = \{U\in \mathcal C : U\neq V\}$,
and denote by $N'(G)$ and $S'(G)$ the number of connected sets and the sum of the sizes of the connected sets, respectively, not including $V$.  For $i \in V$, denote by $G_i$ the graph induced by the vertices $V\setminus \{i\}$.  
Note that $G_i$ is $(k-1)$-connected for all $i\in V$ and therefore, by the
induction hypothesis, we have $S(G_i) \geq a_{k-1} (n-1) N(G_i)$ for all $i\in V$.  
Now
\[\begin{aligned} n S'(G) - \sum_{U\in \mathcal C '} |U|^2 &=    
\sum_{U\in \mathcal C '} |U|(n-|U|) = \sum_{i\in V} S(G_i) \geq a_{k-1} (n-1)
\sum_{i\in V}  N(G_i) \\
&= a_{k-1} (n-1) \sum_{U\in \mathcal C '} (n-|U| ) = a_{k-1} n (n-1) N'(G) -  
a_{k-1} (n-1) S'(G).
\end{aligned}\]
This implies 
\[\big ( n + a_{k-1} (n-1) \big )(S(G)-n) \geq a_{k-1} n (n-1) (N(G)-1)  + \sum_{U\in \mathcal C} |U|^2 - n^2,\]
equivalently
\[\big ( n + a_{k-1} (n-1) \big ) S(G)\geq a_{k-1} n (n-1) N(G) + \sum_{U\in \mathcal C} |U|^2.\]
By Corollary~\ref{cor:1} this implies 
\[ \Big ( \frac{n-1}{2} + a_{k-1} (n-1) \Big ) S(G) \geq a_{k-1} n (n-1) N(G),\]
or
\[D(G) = \frac{S(G)}{n N(G)} = \frac{a_{k-1}}{ \frac{1}{2} + a_{k-1}} 
=  \frac{2 a_{k-1}}{2 a_{k-1}+1}  = a_k.\]

\end{proof}


\begin{thebibliography}{00}

\bibitem{B}  K. J. Balodis, L. Mol, L, and R. Oellermann, On the mean order of connected induced subgraphs of block graphs, 
{\em Australasian J. Comb.} {\bf 76} (2020) 128--148.  

\bibitem{H}  J. Haslegrave, Extremal results on average subtree density of series-reduced trees,
{\em J. Combin. Theory Ser. B} {\bf 107} (2014), 26--41.

\bibitem{H2} J. Haslegrave, The path minimises the average size of a connected set,
arXiv:2103.16491.

\bibitem{J} R. Jamison, On the average number of nodes in a subtree
of a tree.  {\em J. Combin. Theory Ser. B}  {\bf 35}  (1983) 207--223.

\bibitem {J2} R. Jamison,  Monotonicity of the mean order of subtrees, {\em J. Combin. Theory Ser.
B} {\bf 37} (1984), 70--78.

\bibitem {O} M. E. Kroeker,  L. Mol, and O. Oellermann, 
On the mean connected induced subgraph order of
cographs, {\em Australasian J. Combin.}{\bf 71} (2018) 161--183.

\bibitem {MO} L. Mol and O. Oellermann, Maximizing the mean subtree order,  J. Graph Theory,
doi.org/10.1002/jgt.22434, 2018.

\bibitem  {SO} A. M. Stephens and  O. Oellermann, The mean order of sub-$k$-trees of $k$-trees, {\em J. Graph Theory}
{\bf 88} (2018), 61--79.

\bibitem {V} A. Vince and H. Wang, The average order of a subtree of a tree, {\em J CombinTheory Ser B} {\bf 100} (2010) 161--170.

\bibitem {V3}  A. Vince, The average size of a connected vertex set of a graph - explicit formulas and open problems, {\em J. Graph Theory} {\bf 97} (2020), 82--103.  DOI: https://doi.org/10.1002/jgt.22643.

\bibitem {V1} A. Vince, A lower bound on the average size of a connected vertex
set of a graph, arXiv:2103.15174.

\bibitem {WW} S. Wagner and H. Wang, On the local and global means of subtree orders, {\em J. Graph Theory} {\bf 81}
(2016), 154--166.

\bibitem {YY} W. Yan and Y. Yeh, Enumeration of subtrees of trees, {\em Theoret. Comput. Sci.}
{\bf 369} (2006), 256--268.

\end{thebibliography}
\end{document}